\documentclass[psamsfonts,a4paper]{amsart}

\usepackage{amssymb,amsfonts}
\usepackage{amsaddr}
\usepackage[all,arc]{xy}
\usepackage{enumerate}
\usepackage{tikz}
\usepackage{parskip}
\usepackage[margin=2.5cm]{geometry}
\usepackage{eucal}

\newtheorem{thm}{Theorem}[section]

\newtheorem{prop}[thm]{Proposition}
\newtheorem{lemma}[thm]{Lemma}

\theoremstyle{definition}
\newtheorem{defn}[thm]{Definition}

\newtheorem{exmp}[thm]{Example}

\theoremstyle{remark}
\newtheorem{rem}[thm]{Remark}

\makeatletter
\let\c@equation\c@thm
\makeatother
\numberwithin{equation}{section}

\DeclareMathOperator{\der}{Der}

\DeclareMathOperator{\aut}{Aut}

\DeclareMathOperator{\en}{End}

\DeclareMathOperator{\lie}{Lie}
\DeclareMathOperator{\Pic}{Pic}

\DeclareMathOperator{\res}{Res}

\DeclareMathOperator{\spec}{Spec}
\DeclareMathOperator{\str}{STr}
\DeclareMathOperator{\tr}{Tr}

\DeclareMathOperator{\sexp}{sexp}

\newcommand{\CD}{\mathcal{D}}

\newcommand{\bA}{\mathbb{A}}
\newcommand{\C}{\mathbb{C}}

\newcommand{\R}{\mathbb{R}}

\newcommand{\Z}{\mathbb{Z}}

\newcommand{\A}{\mathcal{A}}

\newcommand{\CC}{\mathcal{C}}
\newcommand{\HH}{\mathcal{H}}
\newcommand{\CK}{\mathcal{K}}
\newcommand{\LL}{\mathcal{L}}
\newcommand{\M}{\mathcal{M}}
\newcommand{\OO}{\mathcal{O}}

\newcommand{\V}{\mathcal{V}}

\newcommand{\CF}{\mathcal{F}}

\newcommand{\g}{\mathfrak{g}}

\newcommand{\al}{\alpha}

\newcommand{\G}{\Gamma}

\newcommand{\Om}{\Omega}

\newcommand{\ov}{\overline}

\newcommand{\vac}{{\left|0\right>}}

\newcommand{\coord}{\text{Coord}}
\newcommand{\vir}{\text{Vir}}

\newcommand{\twobytwo}[4]
{\left(\begin{smallmatrix} #1 & #2 \\ #3 & #4 \end{smallmatrix}\right)}

\title{Superconformal Vertex Algebras and Jacobi Forms}
\author{Jethro van Ekeren}
\address{Departamento de Matem\'{a}tica Aplicada, \\ Universidade Federal Fluminense, \\ Rua M\'{a}rio Santos Braga, Niter\'{o}i, RJ 24020-140, Brazil}
\email{jethrovanekeren@gmail.com}

\begin{document}

\begin{abstract}
We discuss the appearance of Jacobi automorphic forms in the theory of superconformal vertex algebras, explaining it by way of supercurves and formal geometry. We touch on related topics such as Ramanujan's differential equations for Eisenstein series.
\end{abstract}

\maketitle

\section{Introduction}

This paper is about the link between automorphic forms and infinite dimensional algebras. It is primarily an exposition of joint work \cite{HVEaccepted} of the author with R. Heluani, which specifically relates certain automorphic forms on the group $SL_2(\Z) \ltimes \Z^2$ called Jacobi forms to vertex algebras equipped with an $N=2$ structure. Being of an expository nature, we have taken the opportunity to make some digressions; in particular to discuss an interpretation (Section \ref{Section.Ramanujan}) of Ramanujan's differential equations as an expression of the `Virasoro uniformisation' of the moduli space of elliptic curves with local coordinate.

The first instance of the aforementioned link was uncovered by Kac and Peterson \cite{KP84}, who used the Weyl-Kac character formula to express characters of integrable modules over affine Kac-Moody algebras in terms of theta functions. Another perspective was later adopted by Zhu \cite{Zhu96}, who proved that characters of suitable conformal vertex algebras are classical modular forms on the group $SL_2(\Z)$. Zhu proceeded by analysing $\CD$-modules associated with the vertex algebra on families of elliptic curves, establishing in particular a certain $SL_2(\Z)$-equivariance.

We study vertex algebras which admit the richer structure of $N=2$ superconformal symmetry. These give rise to $\CD$-modules on families of elliptic supercurves, and we show these $\CD$-modules to be equivariant under a certain $SL_2(\Z) \ltimes \Z^2$-action.

In referring to the $N=2$ superconformal symmetry algebra, whose origin lies in theoretical physics, we mean the Lie superalgebra $\widehat{W}^{1|1}$ with explicit basis and relations as given in Section \ref{Section.SUSY.VA}. To it there is an associated family $L(\widehat{W}^{1|1})_c$ of simple vertex algebras (see Example \ref{Example.SUSY}) depending on an auxiliary parameter $c$ called the central charge. For generic $c$ the representation theory of $L(\widehat{W}^{1|1})_c$ is rather complicated. However for $c(u) = 3 - 6/u$, where $u \in \Z_{\geq 2}$, it turns out that $L(\widehat{W}^{1|1})_{c(u)}$ has precisely $u(u-1)/2$ irreducible modules $L_u(j, k)$, parameterised by the set of pairs $(j, k) \in \Z^2$ where $j \geq 0$, $k \geq 1$ and $j+k < u$. There is an explicit formula for the graded superdimensions of these modules too \cite{Matsuo} \cite{KRW03}, viz.
\begin{align}\label{N2.char.explicit}
\str_{L_u(j, k)} q^{L_0} y^{J_0} = q^{\frac{jk}{u}} y^{\frac{j-k+1}{u}} P_{j, k}^{(u)} / P_{1/2, 1/2}^{(2)},
\end{align}
where
\begin{align*}
P_{j, k}^{(u)} = \prod_{n=1}^\infty \frac{(1-q^{u(n-1)+j+k})(1-q^{un-j-k})(1-q^{un})^2}{(1-q^{un-j}y)(1-q^{u(n-1)+j}y^{-1})(1-q^{un-k}y^{-1})(1-q^{u(n-1)+k}y)}.
\end{align*}




Now the normalised functions $y^{c(u)/6} \str_{L_u(j, k)} q^{L_0} y^{J_0}$ span a vector space which turns out to be invariant under an action of the group $SL_2(\Z) \ltimes \Z^2$ (specifically the weight $0$ index $c/6$ Jacobi action (\ref{Jacobi.action})). In other words the span of the normalised graded superdimensions is a vector valued Jacobi form. The question is to explain this fact conceptually.

In \cite{HVEaccepted} we showed that the picture outlined above relates to a general phenomenon: for any vertex algebra equipped with an `$N=2$ superconformal structure' (of which $L(\widehat{W}^{1|1})_c$ above is an example) the normalised graded superdimensions satisfy $SL_2(\Z) \ltimes \Z^2$-invariant differential equations. The key observations are as follows.
\begin{enumerate}
\item Jacobi forms are essentially sections of vector bundles on the moduli space of pairs $(E, \LL)$, where $E$ is an elliptic curve and $\LL$ a holomorphic line bundle on $E$.

\item Such pairs can be reinterpreted as certain special $1|1$-dimensional supercurves.

\item A vertex algebra $V$ equipped with a suitable $N=2$ superconformal structure `localises' nicely to give a $\CD$-module $\CC$ (known as `conformal blocks') on the moduli space of such supercurves.

\item If $V$ is well-behaved, the normalised graded superdimensions of $V$-modules (generalising the left of (\ref{N2.char.explicit})) converge in the analytic topology and yield horizontal sections of $\CC$.
\end{enumerate}
The issue of convergence is technical. In the appendix to \cite{HVEaccepted} we establish the convergence subject to the well known (to vertex algebraists) condition of $C_2$-cofiniteness. The proof involves analysis of the coefficients of the differential equations corresponding to $\CC$. The key points are to show that these coefficients lie in a certain ring of quasi-Jacobi forms, and to establish that this ring is Noetherian.

For careful statements of results and complete proofs we refer the reader to \cite{HVEaccepted}. 

\emph{Acknowledgements.} I would like to thank the organisers of the 2014 intensive period `Perspectives in Lie Theory' at CRM Ennio De Giorgi, where this work was presented, and CAPES-Brazil for financial support.

\section{Notation}

In addition to standard symbols such as $\C$, $\Z_+ = \{0, 1, 2, \ldots\}$, $\partial_z = \frac{\partial}{\partial z}$, etc., we shall use the following notation without further comment: $\OO = \C[[z]]$ the ring of formal power series in one variable, $\mathfrak{m} = z\C[[z]]$ its maximal ideal, and $\CK = \C((z))$ the ring of Laurent series. The supercommutative algebra $\OO^{1|1}$ is by definition $\OO \otimes \bigwedge[\theta]$, i.e., is obtained by adjoining to $\OO$ a single odd variable $\theta$ satisfying $\theta^2=0$. Similarly we have $\mathfrak{m}^{1|1} = \mathfrak{m} \otimes \bigwedge[\theta]$ and $\CK^{1|1} = \CK \otimes \bigwedge[\theta]$. The structure sheaf, tangent sheaf, cotangent sheaf, and sheaf of differential operators of a (super)scheme $X$ are denoted $\OO_X$, $\Theta_X$, $\Om_X$, and $\CD_X$, respectively.

\section{Superschemes and Elliptic Supercurves}\label{Section.Curves}

The picture to keep in mind of a complex supermanifold (of dimension $m|n$) is that of a space on which the Taylor expansion of a function in terms of local coordinates $z_1, \ldots, z_m$, $\theta_1, \ldots, \theta_n$ lies in the supercommutative ring $\C[[z_i]] \otimes \bigwedge[\theta_j]$. We refer the reader to {\cite{Manin.Gauge}} for background on superalgebra and supergeometry. A superscheme is formally defined {\cite[Chapter 4]{Manin.Gauge}} to be a topological space $X_\text{top}$ together with a sheaf $\OO_X$ of supercommutative local rings, such that the even part $(X_\text{top}, \OO_{X, 0})$ is a scheme. Morphisms are required to be $\Z/2\Z$-graded. The bulk $X_\text{rd}$ of a superscheme $X$ is the scheme $(X_\text{top}, \OO_X / \mathcal{J})$ where $\mathcal{J} = \OO_{X, 1} + \OO_{X, 1}^2$. A (complex) supercurve is a smooth superscheme over $\spec{\C}$ of dimension $1|n$. In this article we shall concern ourselves with $1|1$-dimensional complex supercurves, and we shall generally work in the analytic topology.

Let $X_0$ be a smooth curve and $\LL$ a holomorphic line bundle on $X_0$. We may construct a $1|1$-dimensional supercurve $X$ from this data by putting
\[
\OO_X = \bigwedge \LL[-1] = \OO_{X_0} \oplus \LL,
\]
with $\Z/2\Z$-grading induced by cohomological degree.

Any even family of $1|1$-dimensional complex supercurves is of the above form. Indeed, for a $1|1$-dimensional supercurve defined over a base superscheme $\spec{R}$, transformations between coordinate charts take the general form
\begin{align}\label{gen.11.xform}
\begin{split}
z' &= f_{11}(z) + f_{12}(z) \theta, \\
\theta' &= f_{21}(z) + f_{22}(z) \theta,
\end{split}
\end{align}
where $f_{11}$, $f_{22}$ are power series whose coefficients are even elements of $R$, and $f_{12}$, $f_{21}$ are power series whose coefficients are odd elements of $R$. If the base ring $R$ contains no odd elements then $f_{12}$ and $f_{21}$ vanish, (\ref{gen.11.xform}) is linear in $\theta$ and comprises the \v{C}ech cocycle description of a line bundle $\LL$, and the supercurve is consequently of the form $\bigwedge{\LL[-1]}$.

Recall the set $\Pic(X)$ of isomorphism classes of holomorphic line bundles on a smooth curve $X$, and its subset $\Pic_0(X)$ of line bundles of degree $0$. As is well known {\cite[Appendix B.5]{Hartshorne}} there is a natural bijection $\Pic(X) \cong H^1(X, \OO_X^*)$, and the exponential exact sequence
\[
0 \rightarrow \underline{\Z} \rightarrow \OO_X \rightarrow \OO_X^* \rightarrow 0
\]
yields the following morphisms in cohomology
\[
H^1(X, \underline{\Z}) \rightarrow H^1(X, \OO_X) \rightarrow H^1(X, \OO_X^*) \rightarrow H^2(X, \underline{Z}).
\]
The last map here assigns a line bundle its degree, and the kernel $\Pic_0(X)$ is identified with the quotient
\[
H^1(X, \OO_X) / H^1(X, \underline{\Z})
\]
which is a complex torus of dimension $g$, where $g$ is the genus of $X$.

An elliptic curve is a smooth complex curve of genus $1$, together with a marked point. We shall define an elliptic supercurve to be a supercurve $X$ of dimension $1|1$ whose bulk $X_\text{rd}$ has genus $1$, together with a marked point.

Let $\HH$ denote the complex upper half plane, and let $z$ be the standard coordinate on $\C$ which we fix once and for all. The trivial family $\HH \times \C \rightarrow \HH$ carries the action $(m, n) : (z, \tau) \mapsto (z+m\tau+n, \tau)$ of $\Z^2$ and the quotient together with marked point $z=0$ is a family of elliptic curves, which we denote $E \rightarrow \HH$. 

Quite generally {\cite[Appendix to \S 2]{Mumford.abelian.var}}, for $X$ a topological space with a free discontinuous action of a discrete group $G$, and $\CF$ a sheaf on the quotient space $X/G$ (and with $\pi : X \rightarrow X/G$ the quotient), there is a natural map
\[
H^\bullet(G, \G(X, \pi^* \CF)) \rightarrow H^\bullet(X/G, \CF),
\]
from group cohomology to sheaf cohomology. In case $X$ is a fibre $\C_\tau$ of the trivial family above, this map is an isomorphism. An element $\al \in \C$ defines a group $1$-cocycle $c_\al : \Z^2 \rightarrow \G(\C_\tau, \OO^*)$ by $(m, n) \mapsto e^{2\pi i m \al}$. We denote by $\LL_\al \in \Pic_0(E_\tau)$ the corresponding line bundle on $E_\tau$.

Let $S^\circ = \HH \times \C$. We denote by $E^\circ \rightarrow S^\circ$ the family whose fibre over $(\tau, \al)$ is the elliptic supercurve corresponding to $(E_\tau, \LL_\al)$. The group $SL_2(\Z)$ acts on $E \rightarrow \HH$ in such a way as to identify fibres isomorphic as elliptic curves. We now have the following $1|1$-dimensional analogue.
\begin{prop}\label{Jacobi.action.on.family}
The formulas
\begin{align*}
A : (t, \zeta, \tau, \al) &\mapsto \left( \frac{t}{c\tau+d}, e^{-2\pi i t \frac{c\al}{c\tau+d}}\zeta, \frac{a\tau+b}{c\tau+d}, \frac{\al}{c\tau+d} \right) \\
(m, n) : (t, \zeta, \tau, \al) &\mapsto (t, e^{2\pi i m t} \zeta, \tau, \al+m\tau+n),
\end{align*}
where $A \in SL_2(\Z)$ and $m, n \in \Z$, extend to a left action on $E^\circ \rightarrow S^\circ$ of the semidirect product group
\[
SL_2(\Z) \ltimes \Z^2 \quad \text{where} \quad (A, x) \cdot (A', x') = (AA', xA' + x').
\]
The restriction of the action of $g \in SL_2(\Z) \ltimes \Z^2$ to the fibre $E_{(\tau, \al)}$ is an isomorphism $E_{(\tau, \al)} \cong E_{g \cdot (\tau, \al)}$ of supercurves.
\end{prop}
Every elliptic supercurve associated to an elliptic curve $E$ and its degree $0$ line bundle $\LL$, appears as a fibre of $E^\circ \rightarrow S^\circ$. However $E^\circ \rightarrow S^\circ$ is not a universal family in the sense that it does not `see' families over odd base schemes.

We denote by $\mathbb{A}^{1|1}$ the superscheme whose set of $R$-points is $\spec{R[z, \theta]}$. In fact we distort convention a little by fixing a choice $z, \theta$ of coordinates, in particular our $\mathbb{A}^{1|1}$ has a distinguished origin, and we denote by $(\mathbb{A}^{1|1})^\times$ the subscheme with this origin removed. We then have the algebraic supergroup $GL(1|1)$ of linear automorphisms acting on $(\bA^{1|1})^\times$. The trivial family $(\bA^{1|1})^\times \times GL(1|1) \rightarrow GL(1|1)$ carries the action $n : (x, \mathbf{q}) \mapsto (\mathbf{q}^n x, \mathbf{q})$ of $\Z$. We restrict to the subscheme $S^\bullet \subset GL(1|1)$ consisting of automorphisms with nonzero even reduction, then the quotient by $\Z$ is a family $E^\bullet \rightarrow S^\bullet$ of elliptic supercurves. The distinguished point is $(z, \theta) = (1, 0)$.

We introduce the morphism $\sexp : E^\circ \rightarrow E^\bullet(\C)$ of $\C$-schemes defined by
\begin{align}\label{sex.defined}
(t, \zeta, \tau, \al) \mapsto \left( e^{2\pi i t}, e^{2\pi i t} \zeta, \twobytwo{q}{0}{0}{qy} \right).
\end{align}
The notation $q = e^{2\pi i \tau}$, $y = e^{2\pi i \al}$ used here will be in force throughout the paper.

\begin{rem}
There is a quite distinct notion of supercurve, which we recall here for the sake of avoiding confusion. A $\text{SUSY}_n$ curve {\cite[Chapter 2, Definition 1.10]{Manin.noncomm}} consists of a $1|n$-dimensional supercurve $X$ together with the extra data of a rank $0|n$ subbundle $T \subset \Theta_X$ such that the alternating form
\[
\varphi : T \otimes T \xrightarrow{[\cdot, \cdot]} \Theta_X \longrightarrow \Theta_X / T
\]
is nondegenerate and split. There is a forgetful functor from the category of $\text{SUSY}_n$ curves to that of $1|n$-dimensional supercurves. On the other hand, there turns out to be a nontrivial \emph{equivalence} (due to Deligne {\cite[pp. 47]{Manin.noncomm}}) between the category of all $1|1$-dimensional supercurves, and the category of `orientable' $\text{SUSY}_2$ curves. We describe the correspondence briefly.

Let $(X, T)$ be a $\text{SUSY}_2$ curve. Locally there is a splitting of $T$ as a direct sum of rank $0|1$-subbundles, each isotropic with respect to $\varphi$. If this can be extended to a global splitting, then we say $(X, T)$ is orientable. Suppose this is the case, and let $T_1 \subset T$ be an isotropic subbundle. Set $\ov{X}$ to be the superscheme $(X_\text{top}, \OO_X / T_1 \cdot \OO_X$). Then $\ov{X}$ is a $1|1$-dimensional supermanifold, and $X$ can be recovered uniquely from $\ov{X}$.

Much of the theory discussed below extends straightforwardly to $1|n$-dimensional supercurves, and to $\text{SUSY}_n$ curves.
\end{rem}

\section{The Bundle of Coordinates}

In this section and the two subsequent ones we outline the basics of `formal geometry'. This theory, which goes back to \cite{GeK}, provides a bridge between representation theory of infinite dimensional algebras and geometry of algebraic varieties. The book \cite{FBZ} contains a good introduction for the case of curves. We focus on the case of $1|1$-dimensional supercurves.

The basic object of formal geometry is the `set of all coordinates' on a variety $X$, denoted here by $\coord_X$. It may be defined precisely either as the subscheme of the jet scheme \cite{EM05} consisting of jets with nonzero differential, or as the fibre bundle with fibre at $x \in X$ the set of choices of generator of $\mathfrak{m}_x$ (where $\mathfrak{m}_x$ is the unique maximal ideal of the local ring $\OO_x$ at $x$).

For the case of $X$ a supercurve of dimension $1|1$ we have the noncanonical isomorphism $\OO_x \cong \OO^{1|1}$ at each point $x \in X$. Each fibre therefore carries a simply transitive action of the supergroup $\aut{\OO^{1|1}}$ by changes of coordinates, in other words $\coord_X$ is a principal $\aut{\OO^{1|1}}$-bundle. This supergroup consists of transformations
\begin{align*}
z &\mapsto a_{0, 1} \theta + a_{1, 0} z + a_{1, 1} z \theta + a_{2, 0} z^2 + a_{2, 1} z^2 \theta + \cdots, \\
\theta &\mapsto b_{0, 1} \theta + b_{1, 0} z + b_{1, 1} z \theta + b_{2, 0} z^2 + b_{2, 1} z^2 \theta + \cdots
\end{align*}
where $\twobytwo{a_{0, 1}}{a_{1, 0}}{b_{0, 1}}{b_{1, 0}} \in GL(1|1)$. As such the corresponding Lie superalgebra $\der_0{\OO^{1|1}}$ of derivations preserving $\mathfrak{m}^{1|1}$ has basis
\begin{equation}
\label{n2.plus.basis}
\begin{aligned}
L_n &= -z^{n+1}\partial_z - (n+1)z^n \theta \partial_\theta, &
J_n &= -z^n \theta \partial_\theta, \\
Q_n &= -z^{n+1} \partial_\theta, &
H_n &= z^n \theta \partial_\theta,
\end{aligned}
\end{equation}
where $n \in \Z_+$. The Lie bracket is the usual bracket of vector fields \cite{H07}.

\section{Superconformal Algebras and SUSY Vertex Algebras}\label{Section.SUSY.VA}

The Lie superalgebra $\der_0{\OO^{1|1}}$ embeds into $\der{\CK^{1|1}}$, and we obtain a basis of the latter by extending (\ref{n2.plus.basis}) to $n \in \Z$. This algebra admits a central extension
\[
0 \rightarrow \C C \rightarrow \widehat{W}^{1|1} \rightarrow \der \CK^{1|1} \rightarrow 0,
\]
which splits over $\der_0{\OO^{1|1}}$. Several distinct bases of $\widehat{W}^{1|1}$ appear in the literature. Relative to our choice the explicit relations are as follows {\cite[(2.5.1c)]{HK07}}.
\begin{align*}
[L_m, L_n] &= (m-n) L_{m+n}, &
[L_m, J_n] &= -n J_{m+n} + \delta_{m, -n} \frac{m^2+m}{6} C, & \\
[L_m, H_n] &= -nH_{m+n}, &
[L_m, Q_n] &= (m-n)Q_{m+n}, & \\
[J_m, J_n] &= \delta_{m, -n} \frac{m}{3} C, &
[J_m, Q_n] &= Q_{m+n}, & \\
[J_m, H_n] &= -H_{m+n}, &
[H_m, Q_n] &= L_{m+n} - m J_{m+n} + \delta_{m, -n} \frac{m^2-m}{6} C.
\end{align*}
\begin{rem}
In the $1|0$-dimensional setting we have the analogous Virasoro extension
\[
0 \rightarrow \C C \rightarrow \vir \rightarrow \der{\CK} \rightarrow 0,
\]
with relations
\[
[L_m, L_n] = (m-n) L_{m+n} + \delta_{m, -n} \frac{m^3-m}{12}C.
\]
The naive map $L_n \mapsto L_n$, $C \mapsto C$ is not an embedding of Lie algebras $\vir \hookrightarrow \widehat{W}^{1|1}$, but the map $L_n \mapsto L_n - \frac{1}{2}(n+1)J_n$, $C \mapsto C$ is.
\end{rem}

Though we will not be using vertex algebras until Section \ref{section.cb}, this is a convenient place to give their definition. To avoid clutter we present only the definition of `$N_W=1$ SUSY vertex algebra', which is the variant relevant for us. See \cite{KacVA} and {\cite{HK07}} for the general picture.
\begin{defn}\label{def.SUSY.va}
An $N_W=1$ SUSY vertex algebra is a vector superspace $V$, a vector $\vac \in V$, linear operators $S, T : V \rightarrow V$, and an even linear map $V \otimes V \rightarrow V \widehat{\otimes} \CK^{1|1}$ which is denoted
\[
a \otimes b \mapsto Y(a, Z)b = Y(a, z, \theta)b.
\]
These structures are to satisfy the following axioms.
\begin{enumerate}
\item $Y(\vac, Z) = \text{Id}_V$, and $Y(a, Z)\vac = a \bmod{(V \widehat{\otimes} \mathfrak{m}^{1|1})}$.

\item The series
\[
Y(a, Z) Y(b, W)c, \quad (-1)^{p(a)p(b)} Y(b, W) Y(a, Z)c, \quad \text{and} \quad Y(Y(a, Z-W)b, W)c
\]
are expansions of a single element of $V \widehat{\otimes} \CK^{1|1} \otimes_{\C[z, w]} \C[(z-w)^{-1}]$.

\item $[T, Y(a, Z)] = \partial_z Y(a, Z)$ and $[S, Y(a, Z)] = \partial_\theta Y(a, Z)$.
\end{enumerate}
\end{defn}
The notion of conformal structure (i.e., compatible $\vir$-action) on a vertex algebra permits connection with the geometry of algebraic curves via formal geometry {\cite[Chapter 6]{FBZ}}. Similarly important in the context of $1|1$-dimensional supercurves is the notion of superconformal structure on a SUSY vertex algebra \cite{H07}.
\begin{defn}[{\cite{HK07}}]\label{def.superconf}
A superconformal structure on the $N_W=1$ SUSY vertex algebra $V$ is a pair of vectors $j$ and $h$ (even and odd respectively) such that the following associations furnish $V$ with a $\widehat{W}^{1|1}$-module structure:
\begin{align*}
Y(j, Z) = J(z) - \theta Q(z), \quad Y(h, Z) = H(z) + \theta [L(z) + \partial_z J(z)],
\end{align*}
and
\begin{align*}
J(z) &= \sum_{n \in \Z} J_n z^{-n-1}, & Q(z) &= \sum_{n \in \Z} Q_n z^{-n-2}, \\
H(z) &= \sum_{n \in \Z} H_n z^{-n-1}, & L(z) &= \sum_{n \in \Z} L_n z^{-n-2}.
\end{align*}
It is further required that $T = L_{-1}$, $S = Q_{-1}$, and that $V$ be graded by finite dimensional eigenspaces of $L_0, J_0$, with integral eigenvalues bounded below. For the vector $b \in V$ satisfying $L_0b = \Delta b$, we write $o(b) \in \en{V}$ for the $z^{-\Delta}\theta$ coefficient of $Y(b, z, \theta)$.
\end{defn}
There is a natural notion of module over a SUSY vertex algebra. In the superconformal case we include in the definition $L_0$- and $J_0$-grading conditions analogous to those that appear in Definition \ref{def.superconf}.

\begin{exmp}\label{Example.SUSY}
Let $M(h, m, c)$ denote the Verma module $U(\widehat{W}^{1|1}) \otimes_{U(\widehat{W}^{1|1}_+)} \C{v}$, where the action on $v$ is by $C = c$, $L_0 = h$, $J_0 = m$, $Q_0 = 0$, and all positive modes act by $0$. Let $L(h, m, c)$ denote the unique irreducible quotient of $M(h, m, c)$. Then $M(0, 0, c)$ and $L(\widehat{W}^{1|1})_c = L(0, 0, c)$ have unique superconformal vertex algebra structures such that $v = \vac$, $j = J_{-1}v$, $h = H_{-1}v$.
\end{exmp}

\section{Harish-Chandra Localisation}\label{section.HC}

Let $K$ be a Lie group, $Z$ a principal $K$-bundle on a smooth manifold $S$, and $V$ a left $K$-module. The familiar associated bundle construction produces a vector bundle $\V = Z \times_K V$ on $S$ (recall by definition $Z \times_K V$ is $Z \times V$ modulo the relation $(zg, v) = (z, gv)$). If $\dim S = n$ then $S$ carries a canonical $GL(n)$-bundle, namely the frame bundle, whose fibre at $s \in S$ if the set of all bases of the tangent space $T_sS$. Associated with the defining $GL(n)$-module $\R^n$ is the tangent bundle $\Theta_S$, and with its dual $(\R^n)^*$ the cotangent bundle $\Om_S$.

The functor of Harish-Chandra localisation extends the associated bundle construction, enabling the construction of vector bundles with connection (more properly $\CD$-modules) from $K$-modules with the action of an additional Lie algebra. See {\cite[Chapter 17]{FBZ}} and {\cite[Section 1.2]{BD.Hitchin}} for the general theory.
\begin{defn}
A Harish-Chandra pair $(\g, K)$ consists of a Lie algebra $\g$, a Lie group $K$, an action $\text{Ad}$ of $K$ on $\g$, and a Lie algebra embedding $\lie{K} \hookrightarrow \g$ compatible with $\text{Ad}$. A $(\g, K)$-module is a vector space with compatible left $\g$- and $K$-module structures. A $(\g, K)$-structure on a space $S$ is a principal $K$-bundle $Z \rightarrow S$ together with a transitive action $\g \rightarrow \Theta_Z$ satisfying certain compatibilities.
\end{defn}
Let $Z \rightarrow S$ be a $(\g, K)$-structure, and $V$ a $(\g, K)$-module. The fibre $\V_s$ of the associated bundle $\V = Z \times_K V$ over the point $s \in S$ carries an action of the Lie algebra $\g_s = Z_s \times_K \g$. Inside $\g_s$ we have the pointwise stabiliser $\g_s^0$ of $Z_s$. We denote by $\Delta(V)$ the sheaf whose fibre over $s$ is the space of coinvariants $\V_s / \g_s^0 \cdot \V_s$. The $\g$-action on $V$ translates into a flat connection (more precisely a left $\CD_S$-module structure) on $\Delta(V)$.

Now let $\widehat{\g}$ be a central extension of $\g$ split over $\lie{K} \subset \g$. If $V$ is a $\widehat{\g}$-module then a variation on the construction above yields $\Delta(V)$ a twisted $\CD_S$-module. That is to say, there is a certain sheaf $\CF$ on $S$ (which depends on the central extension $\widehat{\g}$ of $\g$) such that $\Delta(V)$ is a $\CD_\CF$-module, where $\CD_\CF$ is the sheaf of differential operators on $\CF$.

The Harish-Chandra pairs of particular importance in our context are $(\vir, \aut{\OO})$ and $(\widehat{W}^{1|1}, \aut{\OO^{1|1}})$. Their relevance stems from the fact that moduli spaces of curves and $1|1$-dimensional supercurves carry natural $(\g, K)$-structures for these respective pairs {\cite{ADKP}} {\cite{BS88}} (see also {\cite[Chapter 17]{FBZ}} for an overview). This fact frequently goes by the name `Virasoro Uniformisation'.

Let $\pi : X \rightarrow S$ be a morphism of schemes in general. In the sequence
\[
0 \rightarrow \Theta_{X/S} \rightarrow \Theta_X \rightarrow \pi^* \Theta_S \rightarrow 0
\]
(which defines the relative tangent bundle $\Theta_{X/S}$), we denote by $\Theta_\pi$ the preimage of $\pi^{-1} \Theta_S$ in $\Theta_X$. Intuitively $\Theta_\pi$ consists of vector fields on $X$ of the shape $f(s) \partial_s + g(s, x) \partial_x$.

Let $\widehat{\M}$ denote the moduli space of triples $(X, x, t)$ consisting of a smooth algebraic curve $X$ (of genus $g \geq 1$), a point $x \in X$, and a local coordinate $t \in \coord_{X, x}$, and let $\M$ denote the moduli space of pairs $(X, x)$. Let $\pi : \widehat{X} \rightarrow \widehat{\M}$ be the universal curve, and $Y \subset \widehat{X}$ the section of points $(X, x, t; x)$.

The following theorem can be viewed as a refinement of the Kodaira-Spencer isomorphism.
\begin{thm}[{\cite[Lemma 4.1.1]{BS88}}]
There is a canonical $(\der \CK, \aut \OO)$-structure on $\widehat{\M} \rightarrow \M$. It is induced by the $\OO_{\widehat{\M}}$-module isomorphism
\[
\Theta_\pi(\widehat{X} \backslash Y) \rightarrow \OO_{\widehat{\M}} \otimes \der{\CK}
\]
which acts at $(X, x, t)$ by sending a vector field to the expansion at $x$ in powers of $t$ of its vertical component \textup{(}along $X$\textup{)}.
\end{thm}
The $1|1$-dimensional analogue (along with other cases) is {\cite[Theorem 6.1]{Vai95}}. It follows that any $(\vir, \aut{\OO})$-module gives rise to a twisted $\CD$-module on $\M$ (or on any family of smooth curves). Similarly any $(\widehat{W}^{1|1}, \aut{\OO^{1|1}})$-module gives rise to a twisted $\CD$-module on any family of smooth $1|1$-dimensional supercurves.

\section{Elliptic Curves and Ramanujan Differential Equations}\label{Section.Ramanujan}

It is instructive to flesh out the construction of the previous section a little in the case of elliptic curves. Let $E \rightarrow \HH$ be the family of elliptic curves introduced in Section \ref{Section.Curves}, and let $V$ be a $(\der{\CK}, \aut{\OO})$-module, so that we obtain a $\CD$-module $\Delta(V)$ on $\HH$.

We recall some standard functions from number theory \cite{Apostol}. The Bernoulli numbers $B_n$, $n \geq 0$ are defined by
\[
\frac{x}{e^x-1} = \sum_{n=1}^\infty B_n \frac{x^n}{n!}.
\]
The Eisenstein series $G_{2k}$, $k \geq 1$ are defined by
\[
G_{2k} = \frac{(-1)^{k+1}B_{2k}}{(2k)!} (2\pi)^{2k} E_{2k}, \quad \text{where} \quad E_{2k} = 1 - \frac{4k}{B_{2k}} \sum_{m=1}^\infty \frac{n^{2k-1} q^n}{1-q^n}.
\]
The Weierstrass elliptic function $\wp$, and quasielliptic function $\ov{\zeta}$, are defined by
\begin{align*}
\wp(z, \tau) &= z^{-2} + \sum_{k \in \Z_{>0}} (2k-1) z^{2k-2} G_{2k}
\quad \text{and} \quad
-2\pi i \ov{\zeta}(z, \tau) = z^{-1} - \sum_{k \in \Z_{>0}} z^{2k-1} G_{2k}.
\end{align*}
The Weierstrass function $\wp$ is elliptic, i.e.,
\[
\wp(z+1, \tau) = \wp(z+\tau, \tau) = \wp(z, \tau).
\]
The nonstandard normalisation of $\ov{\zeta}$ is chosen so that
\begin{align}\label{zeta.trans}
\ov{\zeta}(z+1, \tau) = \ov{\zeta}(z, \tau) \quad \text{and} \quad
\ov{\zeta}(z+\tau, \tau) = \ov{\zeta}(z, \tau)+1.
\end{align}
We have the following result.
\begin{lemma}\label{global.vec.field}
Flat sections $s$ of $\Delta(V)$ satisfy the differential equation
\[
\frac{\partial s}{\partial \tau} + \left(\res_t \ov{\zeta}(t, \tau) L(t) dt\right) \cdot s = 0.
\]
\end{lemma}

\begin{proof}
The proof is an exercise in unwinding the definitions of Section \ref{section.HC}, applied to $E \rightarrow \HH$. All that needs to be checked is that the vector field $\partial_\tau + \ov{\zeta}(z, \tau) \partial_z$ is well defined on $E$ (being \emph{a priori} well defined only on its universal cover, since $\ov{\zeta}$ is not elliptic).

Under the transformation $(z', \tau') = (z+\tau, \tau)$ we have $\partial_{\tau'} = \partial_\tau - \partial_z$ and $\partial_{z'} = \partial_z$. This together with (\ref{zeta.trans}) shows that $\partial_\tau + \ov{\zeta}(z, \tau) \partial_z$ is well defined. The same check on the transformation $(z, \tau) \mapsto (z+1, \tau)$ is immediate.
\end{proof}
Another incarnation of Lemma \ref{global.vec.field} is the following partial differential equation satisfied by the Weierstrass function $\wp$.
\begin{prop}
The Weierstrass functions satisfy
\begin{align}\label{Weier.DE}
\frac{\partial}{\partial \tau} \wp  + \ov{\zeta} \frac{\partial}{\partial z}\wp = \frac{1}{2\pi i}(2\wp^2 - 2 G_2 \wp - 20 G_4).
\end{align}
\end{prop}

\begin{proof}
Differentiating
\[
\wp(z+\tau, \tau) - \wp(z, \tau) = 0
\]
with respect to $\tau$ yields
\[
\dot{\wp}(z+\tau, \tau) - \dot{\wp}(z, \tau) = -\wp'(z+\tau, \tau)
\]
(where $\wp'$ and $\dot{\wp}$ are the derivatives with respect to the first and second entries). Similarly $\dot{\wp}(z+1, \tau) - \dot{\wp}(z, \tau) = 0$. It is clear then that $\dot{\wp} + \ov{\zeta} \wp'$ is an elliptic function with pole of order $4$ at $z=0$, hence a polynomial in $\wp$. Comparing leading coefficients yields the result.
\end{proof}
Equating coefficients of (\ref{Weier.DE}) yields an infinite list of differential equations on Eisenstein series. The first three of these, viz.
\begin{align*}
q \partial{E_2}/\partial q &= (E_2^2 - E_4)/12, \\
q \partial{E_4}/\partial q &= (E_2 E_4 - E_6)/3, \\
q \partial{E_6}/\partial q &= (E_2 E_6 - E_4^2)/2,
\end{align*}
were discovered by Ramanujan \cite{ramanujan} (see also \cite{vdP} and \cite{Movasati2012}).

\section{Conformal Blocks and Trace Functions}\label{section.cb}

A conformal vertex algebra carries a $(\vir, \aut{\OO})$-module structure, so the machinery of Section \ref{section.HC} can be applied. Let $V$ be a conformal vertex algebra and $X$ a smooth algebraic curve, we obtain an associated bundle $\V = \coord_X \times_{\aut{\OO}} V$ on $X$. Put $\A = \V \otimes \Om_X$. The vertex operation on $V$ has not yet been used, it translates into the following structure on $\A$: for each $x \in X$ an action $\mu$ of the space of sections $\G(D_x^\times, \A)$ on the fibre $\A_x$ (here $D_x^\times = \spec{\CK_x}$ is the punctured infinitesimal disc at $x$). In fact this structure makes $\A$ into a chiral algebra on $X$, in the sense of \cite{BD} (see also {\cite[Theorem 19.3.3]{FBZ}}). Underlying this construction is the following formula due to Huang {\cite{HuangCFT}}
\begin{align}\label{Huang.lemma}
R(\rho) Y(a, z) R(\rho)^{-1} = Y(R(\rho_z) a, \rho(z)),
\end{align}
valid for all $\rho \in \aut{\OO}$. Here $\rho_z \in \aut{\OO}$ is the automorphism defined by $\rho_z(t) = \rho(z+t)-\rho(t)$, and $R(\rho)$ is the action of $\rho$ on the conformal vertex algebra $V$ (obtained by exponentiating $\der_0{\OO} \subset \vir$).

Applying the Harish-Chandra formalism to $V$ and to a family $X, x$ of pointed curves over base $S$ yields the $\CD_S$-module $\Delta(V)$, with fibres
\[
\frac{\A_x}{\G(X \backslash x, \A) \cdot \A_x}.
\]
The dual of this fibre is called the vector space of conformal blocks associated with $X, x, V$, and is denoted $\CC(X, x, V)$.

A superconformal SUSY vertex algebra carries a $(\widehat{W}^{1|1}, \aut{\OO^{1|1}})$-module structure, and can therefore be similarly localised on $1|1$-dimensional supercurves. These sheaves are again chiral algebras, using {\cite[Theorem 3.4]{H07}} which is a general SUSY analogue of (\ref{Huang.lemma}) above.

The theorems of this section and the next concern construction of horizontal sections of the conformal blocks bundle $\CC$ for elliptic supercurves, and the modular properties of these sections. They are super-analogues of fundamental results of Zhu {\cite{Zhu96}}.
\begin{thm}[{\cite[Proposition 7.10]{HVEaccepted}}]\label{triscb}
Let $V$ be a superconformal vertex algebra and $M$ its module. Let $X = (\mathbb{A}^{1|1})^\times / \mathbf{q}$ be an elliptic supercurve with marked point $x = (z, \theta) = (1, 0)$ as in Section \ref{Section.Curves}. Then the element of $V^*$ defined by
\[
\varphi_M : b \mapsto \str_M o(b) R(\mathbf{q})
\]
is a conformal block, i.e., $\varphi_M \in \CC(X, x, V)$.
\end{thm}

\begin{proof}[Sketch]
Let $a, b$ be sections of a chiral algebra $\A, \mu$ on $(\mathbb{A}^{1|1})^\times$. Huang's formula (\ref{Huang.lemma}) can be written schematically as
\[
\rho \mu(a) \rho^{-1} = \mu(\rho \cdot a).
\]
Item (2) of Definition {\ref{def.SUSY.va}} may be reformulated {\cite[Theorem 3.3.17]{HK07}} as the relation
\[
\mu(a) \mu(b) - \mu(b) \mu(a) = \mu(\mu(a)b)
\]
(again expressed only schematically). Let $\mathbf{q} \in GL(1|1)$, and suppose $a$ is $\mathbf{q}$-equivariant. Then the relations above combine with (super)symmetry of the (super)trace, and equivariance of $a$, to yield
\begin{align}\label{main.loop}
\begin{split}
\tr \mu(\mu(a)b) \mathbf{q}
&= \tr \left[\mu(a) \mu(b) - \mu(b) \mu(a)\right] \mathbf{q} \\
&= \tr \left[\mu(a) \mu(b) \mathbf{q} - \mu(b) \mathbf{q} \mu(\mathbf{q} \cdot a)\right] \\
&= \tr \left[\mu(a) \mu(b) \mathbf{q} - \mu(b) \mathbf{q} \mu(a)\right] \\
&= \tr \left[\mu(a) \mu(b) \mathbf{q} - \mu(a) \mu(b) \mathbf{q}\right] \\
&= 0.
\end{split}
\end{align}
In other words $b \mapsto \str \mu(b) \mathbf{q}$ annihilates the action of global $\mathbf{q}$-equivariant sections, and hence is a conformal block on $(\mathbb{A}^{1|1})^\times / \mathbf{q}$. This sketch can be made precise either in the language of chiral algebras or of vertex algebras (and is done so in {\cite{HVEaccepted}} Sections 7.10 and 7.11, respectively).
\end{proof}
In \cite{Zhu96} the convergence in the analytic topology of the series defining $\varphi_M$ is important, and is derived from a finiteness condition on $V$ called $C_2$-cofiniteness. The superconformal analogue is proved in {\cite[Appendix A]{HVEaccepted}}, also using $C_2$-cofiniteness.

We may now regard the element $\varphi_M \in \CC(E^\bullet(\mathbf{q}), V)$ as a section of the sheaf $\CC$ of conformal blocks on $S^\bullet$. As we have seen this sheaf is a twisted $\CD$-module. The $\varphi_M$ are flat sections of $\CC$, as we shall see in the next theorem via a variation on the proof of Theorem \ref{triscb}. Though the argument applies generally, we restrict attention to even $\mathbf{q} = \twobytwo{q}{0}{0}{qy}$ for the sake of clarity. In this case the operator $R(\mathbf{q})$ on $V$ is simply $q^{L_0} y^{J_0}$ and we recover the supercharacter
\begin{align}\label{supercharacter}
\varphi_M(b) = \str_M o(b) q^{L_0} y^{J_0}.
\end{align}

Expressed in terms of $x = e^{2\pi i t}$ we have the following expression for the Weierstrass function
\begin{align}\label{x.express.zeta}
\ov{\zeta}(t, \tau) = \xi(x, q) = \frac{1}{2} + \frac{1}{x-1} + \sum_{n \in \Z \backslash 0} \left( \frac{1}{q^n x - 1} - \frac{1}{q^n-1} \right).
\end{align}
We remark that the relation $\xi(qx, q) = \xi(x, q)+1$ is easily deduced from (\ref{x.express.zeta}) via a telescoping sum argument.
\begin{thm}[{\cite[Theorem 8.15]{HVEaccepted}}]\label{trhasode}
The function $\varphi_M$ satisfies the following \textup{(}in general infinite\textup{)} system of PDEs:
\begin{align*}
q \frac{\partial}{\partial q} \varphi_M(b) &= \varphi_M(\res_{x=1} x \xi(x, q) L(x-1) b), \\
y \frac{\partial}{\partial y} \varphi_M(b) &= \varphi_M(\res_{x=1} \xi(x, q) J(x-1) b).
\end{align*}
\end{thm}

\begin{proof}[Sketch]
As in Theorem \ref{triscb} we work on $(\mathbb{A}^{1|1})^\times$ with coordinates $(z, \theta)$ fixed. We let $a$ be a section of $\A$ no longer $\mathbf{q}$-equivariant, but satisfying instead $\mathbf{q} \cdot a = a - s$, where $s$ will be one of the explicit sections $h z$ or $j \theta$. We repeat the calculation (\ref{main.loop}) to obtain
\begin{align*}
\str \mu(\mu(a)b) \mathbf{q} = \str \mu(b) \mathbf{q} \mu(s)
\end{align*}
in general.

The function $\xi$ may be used to construct the appropriate section $a$ because of the key relation $\xi(qx) = \xi(x)+1$. A precise calculation (in, for instance, the case $s = hz$) yields
\begin{align*}
\varphi_M(\res_{x=1} x \xi(x, q) L(x-1) b) = \str_M o(b) q^{L_0} y^{J_0} L_0 = q \frac{\partial}{\partial q} \varphi_M(b).
\end{align*}
The other relation derives in the same way from $s = j \theta$.
\end{proof}

\begin{rem}\label{connec.de}
By the same reasoning as in Lemma \ref{global.vec.field}, we see that the differential equations of Theorem \ref{trhasode} are essentially the explicit expressions of the canonical Harish-Chandra connection.
\end{rem}

\section{Jacobi Modular Invariance}

We now study the pullbacks of the sections $\varphi_M$ via the morphism $\sexp$ defined by formula (\ref{sex.defined}). We show that (after a normalisation) they are horizontal with respect to a certain $SL_2(\Z) \ltimes \Z^2$-equivariant connection. Explicitly we prove the following result.
\begin{thm}[{\cite[Theorem 9.10]{HVEaccepted}}]\label{equivar.of.sections}
The normalised section
\[
\widetilde{\varphi}_M = e^{2\pi i \al \cdot (C/6)} \sexp^*(\varphi_M)
\]
is flat with respect to the connection
\[
\nabla = d + \left( \res_t \ov{\zeta}(t, \tau) J(t) dt \right) d\al + \frac{1}{2\pi i} \left( \res_z \ov{\zeta}(z, \tau) \left[ L(z) + \partial_z J(z) \right] \right) d\tau.
\]
Furthermore $\nabla$ is equivariant with respect to the $SL_2(\Z) \ltimes \Z^2$-action on $E^\circ \rightarrow S^\circ$ of Proposition \ref{Jacobi.action.on.family}.
\end{thm}
This theorem is proved by analysing the behaviour of the partial differential equations of Theorem \ref{trhasode} under $SL_2(\Z) \ltimes \Z^2$ transformations, which is an explicit computation.

It is possible to write the (projective) $SL_2(\Z) \ltimes \Z^2$-action on flat sections $\widetilde{\varphi}$ of $\CC$ explicitly {\cite[Theorem 1.2 (c)]{HVEaccepted}}. The specialisation to $b = \vac$ is
\begin{align}\label{Jacobi.action}
\begin{split}
[\widetilde{\varphi} \cdot (m, n)](\vac, \tau, \al) &= \exp{2\pi i \frac{C}{6} \left[ m^2 \tau + 2m\al + 2n \right]} \widetilde{\varphi}(\vac, \tau, \al+m\tau+n) \\
[\widetilde{\varphi} \cdot \twobytwo abcd](\vac, \tau, \al) &= \exp{2\pi i \frac{C}{6} \left[ \frac{-c\al^2}{c\tau+d} \right]} \widetilde{\varphi}\left( \vac, \frac{a\tau+b}{c\tau+d}, \frac{\al}{c\tau+d} \right).
\end{split}
\end{align}
This recovers the well known transformation law {\cite[Theorem 1.4]{EZ}} for Jacobi forms of weight $0$ and index $C/6$. Evaluation at other elements $b \in V$ yields Jacobi forms of higher weight, as well as more complicated `quasi-Jacobi' forms.

In order to deduce Jacobi invariance of the (normalised) supercharacters (\ref{supercharacter}) it suffices to show that they span the fibre of $\CC$. This can presumably be done following the method of Zhu \cite[Section 5]{Zhu96}, assuming $V$ is a rational vertex algebra. Alternatively Jacobi invariance can be proved by extending the calculations of \cite{KM13} to the supersymmetric case.

\bibliographystyle{plain}

\def\cprime{$'$}

\end{document}